\documentclass[11pt]{article}

\usepackage{hyperref}
\usepackage{latexsym}
\usepackage{amssymb}
\usepackage{amsthm}
\usepackage{amscd}
\usepackage{amsmath}
\usepackage{tikz}
\usepackage{youngtab}

\usepgflibrary{arrows}
\usetikzlibrary[positioning,patterns]

\newtheorem{theorem}{Theorem}[section]

\newtheorem{lemma}[theorem]{Lemma}

\newtheorem{proposition}[theorem]{Proposition}
\newtheorem{corollary}[theorem]{Corollary}

\theoremstyle{definition}
\newtheorem*{example}{Example}
\newtheorem*{remark}{Remark}

\numberwithin{equation}{section}

\newcommand{\CC}{\mathbb{C}}
\newcommand{\FF}{\mathbb{F}}

\newcommand{\QQ}{\mathbb{Q}}
\newcommand{\ZZ}{\mathbb{Z}}

\newcommand{\cA}{\mathcal{A}}
\newcommand{\cB}{\mathcal{B}}

\newcommand{\cP}{\mathcal{P}}

\newcommand{\cS}{\mathcal{S}}
\newcommand{\cT}{\mathcal{T}}
\newcommand{\cU}{\mathcal{U}}

\newcommand{\cX}{\mathcal{X}}

\newcommand{\fkgl}{\mathfrak{gl}}

\newcommand{\Hom}{\mathrm{Hom}}

\newcommand{\Ind}{\mathrm{Ind}}
\newcommand{\GL}{\mathrm{GL}}

\newcommand{\Inf}{\mathrm{Inf}}
\newcommand{\ch}{\mathrm{ch}}
\newcommand{\Irr}{\mathrm{Irr}}

\newcommand{\sh}{\mathrm{sh}}

\newcommand{\diag}{\mathrm{diag}}

\newcommand{\Id}{\mathrm{Id}}

\newcommand{\UT}{\mathrm{UT}}

\newcommand{\Sym}{\mathrm{Sym}}

\newcommand{\scf}{\mathrm{scf}}

\newcommand{\dd}{\displaystyle}
\newcommand{\scs}{\scriptstyle}
\newcommand{\scscs}{\scriptscriptstyle}

\newcommand{\vphi}{\varphi}

\newcommand{\spanning}{\textnormal{-span}}
\newcommand{\One}{{1\hspace{-.14cm} 1}}

\newcommand{\larc}[1]{\hspace{-.4ex}\overset{#1}{\frown}\hspace{-.4ex}}
\newcommand{\slarc}[1]{\overset{#1}{\frown}}
\def\adots{\mathinner{\mkern2mu\raise0pt\hbox{.}  
\mkern2mu\raise4pt\hbox{.}\mkern1mu
\raise7pt\vbox{\kern7pt\hbox{.}}\mkern1mu}}

\newcommand{\bl}{\mathrm{bl}}

\newcommand{\semisimple}{\mathrm{ss}}
\newcommand{\unipotent}{\mathrm{un}}
\newcommand{\cf}{\mathrm{cf}}
\newcommand{\uscf}{\mathrm{cf}^{\mathrm{un}}_{\mathrm{supp}}}
\newcommand{\ucf}{\mathrm{cf}^{\mathrm{un}}_{\mathrm{char}}}

\newcommand{\GGG}{\mathrm{ggg}}
\newcommand{\nn}{\mathrm{nn}}
\newcommand{\ctr}{\mathrm{ctr}}
\newcommand{\circast}{\bigcirc\hspace{-.3cm}\ast}

\newcommand{\comment}[1]{\ }

\tikzstyle{bsq}=[rectangle, draw, thick, minimum width=.75cm, minimum height=.75cm]

\makeatletter
\renewcommand{\@makefnmark}{\mbox{\textsuperscript{}}}
\makeatother

\allowdisplaybreaks[1]

\begin{document}

\title{The combinatorics of $\GL_n$\\ generalized Gelfand--Graev characters}
\author{Scott Andrews\\ Department of Mathematics\\ Dartmouth College \\ \textsf{scott.d.andrews@dartmouth.edu} \and Nathaniel Thiem\\ Department of Mathematics\\ University of Colorado \textbf{Boulder}\\
 \textsf{thiemn@colorado.edu}}

\date{}

\maketitle

\begin{abstract}
Introduced by Kawanaka in order to find the unipotent representations of finite groups of Lie type, generalized Gelfand--Graev characters have remained somewhat mysterious.  Even in the case of the finite general linear groups, the combinatorics of their decompositions has not been worked out.  This paper re-interprets Kawanaka's definition in type $A$ in a way that gives far more flexibility in computations.  We use these alternate constructions to show how to obtain generalized Gelfand--Graev representations directly from the maximal unipotent subgroups.   We also explicitly decompose the corresponding generalized Gelfand--Graev characters in terms of unipotent representations, thereby recovering the Kostka--Foulkes polynomials as multiplicities.
\end{abstract}

\section{Introduction}

There has been considerable progress in recent years on the combinatorial representation theory of finite unipotent groups (some recent examples include \cite{MR3117506,MR3323980, MR3295975, MR3313496}).  For example, the representation theory of the maximal unipotent subgroup $\UT_n(\FF_q)$ of the finite general linear group $\GL_n(\FF_q)$ has developed from a wild problem to a combinatorial theory based on set partitions \cite{MR1338979,yan}.  Furthermore, by gluing together these theories we get a Hopf structure analogous to the representation theory of the symmetric groups $S_n$ (where we replace the symmetric functions of $S_n$ with symmetric functions in non-commuting variables for $\UT_n(\FF_q)$) \cite{MR2880223}.

An  underlying philosophy of this paper is that the Bruhat decomposition of a finite group of Lie type 
$$G=\bigsqcup_{w\in W\atop t\in T}UtwU$$ 
gives a factorization into a maximal unipotent $U$-part, a torus $T$-part , and a Weyl group $W$-part.  The traditional approach to studying the representation theory of these groups has been to tease out the influence of the representation theory of $W$ and 
$T$.  With the representation theory of $U$ well-known to be wild, this approach seemed natural and eventually led to Lusztig's classification of the irreducible representations of $G$ \cite{MR742472}.  

Lusztig's indexing, however, is not overly constructive.  In particular, we would like combinatorial constructions of the unipotent representations of $G$, and here the representation theory of $U$ has untapped potential.   The most natural way to find unipotent representations of $G$ is to induce representations from $U$.  There are a number of known approaches:
\begin{description}
\item[(GG)]  The Gelfand--Graev representation is obtained by inducing a linear representation of $U$ in general position;
\item[(DGG)] The degenerate Gelfand--Graev representations generalize the GG representations by inducing arbitrary linear representations of $U$;
\item[(GGG)]  The generalized Gelfand--Graev representations provide a different generalization by instead inducing certain linear representations in general position from specified subgroups of $U$.
\end{description}
The GG representations were introduced to find cuspidal representations of $G$, and it was hoped that the DGG representations could identify all of the unipotent representations of $G$.  While this works for $\GL_n(\FF_q)$ \cite{MR643482}, in general the DGG representations are insufficient.   Kawanaka introduced the GGG representations \cite{MR803335} as a more effective method; the trade-off is that they appear to be more difficult to work with.  

GGG representations have recently seen more attention, especially through the work of Taylor \cite{MR3081621} and Geck--H\'ezard \cite{MR2468596}.  Their work approaches these representations via Deligne--Lusztig theory, whereas this paper will focus on directly on the underlying unipotent groups.   Even for $\GL_n(\FF_q)$ these representations are not particularly well-understood, and this paper hopes to develop this example as a model for tackling other types.

Each GGG representation $\Ind_{U'}^{\GL_n(\FF_q)}(\gamma)$ is induced from a linear representation $\gamma$ of a subgroup $U'\subseteq \UT_n(\FF_q)$.  The construction given by Kawanaka uses the root combinatorics of the corresponding Lie algebra to identify $U'$ and $\gamma$; however, inducing makes these specific choices somewhat artificial.  Our main result of Section \ref{SectionGGGDefinition} (Corollary \ref{GGGSufficiency}) gives a more direct set of sufficient conditions for pairs $(\gamma, U')$ that induce to GGG representations.

An alternative approach is to identify representations of $\UT_n(\FF_q)$ that induce to GGG representations;  that is, are there choices for $(\gamma,U')$ such that we already know the $\UT_n(\FF_q)$-module $\Ind_{U'}^{\UT_n(\FF_q)}(\gamma)$?  Our  main results of Section \ref{SectionSupercharacters} (Corollaries \ref{GGGConstruction} and \ref{GGGMult1Construction}) show that these induced representations may in fact be chosen so that they afford supercharacters of a natural supercharacter theory \cite{MR2373317} of $\UT_n(\FF_q)$; this supercharacter theory, which is built on non-nesting set partitions, is described by Andrews in \cite{andrews3}. From this point of view, we could conduct all our constructions using known monomial representations of $\UT_n(\FF_q)$.

Given our understanding of GGG representations from the previous sections, we decompose the corresponding characters into unipotent characters of 
$\GL_n(\FF_q)$ using Green's symmetric function description \cite{MR0072878,MR1354144}.  Our main result of Section \ref{SectionSymmetricFunctions} (Theorem \ref{unipotentmultiplicity}) is that the multiplicities of the unipotent characters are given exactly by Kostka--Foulkes polynomials; this effectively makes the GGG characters $q$-analogues of the DGG characters, and gives another representation theoretic interpretation for these polynomials.   This result also implies that the projections of the GGG characters to  the unipotent characters are the transformed Hall--Littlewood polynomials.

We consider this paper to be the first steps in a larger program of constructing unipotent modules for finite groups of Lie type.  In \cite{andrews4}, Andrews uses the constructions of this paper to explicitly construct the unipotent modules for $\GL_n(\FF_q)$.  However, for other types there is more work to do.  For type $C$ we have an idea of what the analogues of Corollaries \ref{GGGConstruction} and \ref{GGGMult1Construction}  should be, but even for type $B$ there is again more work.  We hope this paper can give a road map for future constructions.

\vspace{.5cm}

\noindent\textbf{Acknowledgements.}  The authors would like to thank Sami Assaf for pointing us to Haiman's description of transformed Hall--Littlewood polynomials.  Thiem was in part supported by NSA H98230-13-1-0203.

\section{Preliminaries}

This section introduces some of the background topics for the paper.  In particular, we give brief introductions to the representation theory of the finite general linear groups, our set partition combinatorics, the unipotent subgroups of greatest interest, and the notion of a supercharacter theory.

\subsection{The combinatorial representation theory of $\GL_n(\FF_q)$}\label{GLnCombinatorics}

In this section we present an indexing of the irreducible characters of $\GL_n(\FF_q)$ and describe these characters in terms of symmetric functions. This result is initially due to Green \cite{MR0072878} and is also found in \cite{MR1354144}.

If $n = mk$, $\varphi \in  \Hom(\FF_{q^k}^\times, \CC^\times)$, and $x \in \FF_{q^n}^\times$, define an injective homomorphism
$$\begin{array}{c@{\ }c@{\ }c}  \Hom(\FF_{q^k}^\times, \CC^\times) & \longrightarrow &  \Hom(\FF_{q^n}^\times, \CC^\times)\\ \varphi & \mapsto & \vphi\circ N_{\FF_{q^n}/\FF_{q^k}}\end{array}
\quad \text{where}\quad 
\begin{array}{r@{\ }c@{\ }c@{\ }c}N_{\FF_{q^n}/\FF_{q^k}} :&\FF_{q^n} & \longrightarrow & \FF_{q^k}\\ & t & \mapsto & t^{1+q^m+q^{2m}+\cdots+q^{(k-1)m}}.\end{array}
$$
With these identifications, let
\[
        \hat{\bar\FF}_q^\times = \bigcup_{n \geq 1}  \Hom(\FF_{q^n}^\times, \CC^\times).
\]

Let $\sigma:\bar{\FF}_q^\times \to \bar{\FF}_q^\times$ be the Frobenius map defined by $\sigma(x) = x^q$. The group $\langle \sigma \rangle$ acts on $\hat{\bar\FF}_q^\times$ by $\sigma(\varphi)(x) = \varphi(\sigma(x))$. Let
\begin{equation*}
        \Theta = \{ \text{$\langle\sigma\rangle$-orbits in $\hat{\bar\FF}_q^\times$}\} \qquad \text{and}\qquad \Phi  = \{\text{$\langle\sigma\rangle$-orbits in $\bar{\FF}_q^\times$}\}.
\end{equation*}
If $\cP$ is the set of integer partitions and $\cX$ is a set, define the set of \emph{$\cX$-partitions} $\cP^\cX$ to be the set
$$\cP^\cX=\bigg\{\begin{array}{r@{\ }c@{\ }c@{\ }c}\lambda: & \cX & \longrightarrow & \cP\\ & \vphi & \mapsto & \lambda^{(\vphi)}\end{array}\bigg\}.$$
 For $\cX\in \{\Theta,\Phi\}$, the \emph{size}  of $\lambda \in \cP^\cX$ is
\[
        |\lambda| = \sum_{\varphi \in \cX} |\varphi||\lambda^{(\varphi)}|. 
\]

\begin{theorem}[{\cite[Theorem 14]{MR0072878}}]
The complex irreducible characters of $\GL_n(\FF_q)$ are indexed by the functions $\lambda \in \cP^\Theta$ such that $|\lambda| = n$, and the conjugacy classes of $\GL_n(\FF_q)$ are indexed by the functions $\mu \in \cP^\Phi$ such that $|\mu| = n$.
\end{theorem}

The work of Green \cite{MR0072878} allows us to realize the characters of $\GL_n(\FF_q)$ as symmetric functions. For background on the space of symmetric functions, see \cite{MR1354144}; we will only present the material that is necessary for our results.

The $\CC$-vector space
\[
     \cf(\GL) = \bigoplus_{n \geq 1} \cf(\GL_n),\quad \text{where}\quad   \cf(\GL_n) = \{\text{class functions of $\GL_n(\FF_q)$}\},
\]
has a graded commutative $\CC$-algebra structure with multiplication given by parabolic induction: for $\chi\in \cf(\GL_m)$ and $\psi\in \cf(\GL_n)$,
\begin{equation}\label{ParabolicInduction}
        \psi \circ \chi = \Ind_{P_{(m,n)}}^{\GL_{m+n}(\FF_q)}\Inf_{L_{(m,n)}}^{P_{(m,n)}} (\psi \times \chi),
\end{equation}
where
$$L_{(m,n)}=\left[\begin{array}{@{}c|c@{}} \GL_m(\FF_q) & 0\\ \hline 0 &\GL_n(\FF_q) \end{array}\right]\subseteq P_{(m,n)}=\left[\begin{array}{@{}c|c@{}} \GL_m(\FF_q) & \ast\\ \hline 0 &\GL_n(\FF_q) \end{array}\right].$$

For each $f \in \Phi$, let $X^{(f)}=\{X_1^{(f)},X_2^{(f)},\ldots\}$ be a countably infinite set of variables. We define
$$\Sym(\GL)=\bigotimes_{f\in \Phi} \Sym(X^{(f)}),$$
where $\Sym(X^{(f)})$ is the $\CC$-algebra of symmetric functions
\begin{align*}
        \Sym(X^{(f)}) &= \CC\spanning\{p_\nu(X^{(f)}) \mid \nu\in\cP\}\\
        &= \CC\spanning\{s_\nu(X^{(f)}) \mid \nu\in\cP\},
        \end{align*}
written in terms of the power-sum basis and the Schur function basis, respectively.   
     
Note that we have a grading
$$\Sym(\GL)=\bigoplus_{n\geq 0} \Sym_n(\GL),\quad \text{where}\quad \Sym_n(\GL)=\CC\spanning\left\{\left.p_\mu=\prod_{f\in\Phi} p_{\mu^{(f)}}(X^{(f)})\;\right|\; |\mu|=n\right\}.$$

For $\lambda=(\lambda_1,\ldots,\lambda_\ell)$ a partition of $k$ and $X$ an infinite set of variables, define
\[
        \widetilde{P}_\lambda(X;q) = q^{-n(\lambda)}P_{\lambda}(X;q^{-1}), \quad\text{where}\quad n(\lambda) = \sum_{i=1}^\ell (i-1)\lambda_i,
\]
and $P_{\lambda}(X;q)$ is the Hall--Littlewood symmetric function. For $\mu \in \cP^\Phi$, define
\[
        \widetilde{P}_\mu(X;q) = \prod_{f \in \Phi} \widetilde{P}_{\mu^{(f)}}(X^{(f)};q^{|f|}).
\]
Each $\mu\in \cP^\Phi$ corresponds to a conjugacy class of $\GL_{|\mu|}(\FF_q)$; we define
the indicator functions $\delta_\mu:\GL_{|\mu|}(\FF_q) \rightarrow \CC$ by
$$\delta_\mu(g)=\left\{\begin{array}{@{}ll} 1 & \text{if $g$ has conjugacy type $\mu$,}\\ 0 & \text{otherwise.}\end{array}\right.$$

\begin{theorem}[{\cite[IV.4.1]{MR1354144}}]
The characteristic  function
$$\begin{array}{r@{\ }c@{\ }c@{\ }c}
        \ch :& \cf(\GL) & \longrightarrow & \Sym(\GL) \\
            &\delta_\mu & \mapsto  &\widetilde{P}_\mu(X;q).
\end{array}
$$
is an isomorphism of graded $\CC$-algebras.
\end{theorem}

To describe the images of the irreducible characters of $\GL_n(\FF_q)$ under the characteristic map, we introduce a new set of variables. For each $\varphi \in \Theta$, let $Y^{(\varphi)}=\{Y^{(\vphi)}_1,Y^{(\vphi)}_2,\ldots\}$ be the countably infinite set of variables completely determined by 
\begin{equation}\label{VariableRelation}
        p_k(Y^{(\varphi)}) = (-1)^{k|\varphi|-1}\sum_{f \in \Phi\atop f\subseteq \FF_{q^{k|\varphi|}}} \bigg(\sum_{x\in f} \varphi(x) \bigg) p_{\frac{k|\varphi|}{|f|}}(X^{(f)}),\quad \text{for $k\in \ZZ_{\geq 1}$}.
\end{equation}

For $\lambda$ a partition of $n$, let $s_\lambda(Y)$ denote the Schur function. For $\nu \in \cP^\Theta$, let
\begin{equation} \label{IrreducibleGLCharacters}
        s_\nu = \prod_{\varphi \in \Theta}s_{\nu^{(\varphi)}}(Y^{(\varphi)}).
\end{equation}

\begin{theorem}[{\cite[IV.6.8]{MR1354144}}] The set 
$$\{\ch^{-1}(s_\nu)\mid \nu\in \cP, |\nu|=n\}$$
is exactly the set of irreducible characters of $\GL_n(\FF_q)$.
\end{theorem}
In light of this result, for $\nu \in \cP^{\Theta}$, let 
$$\chi^{\nu} = \ch^{-1}(s_{\nu}).$$

Let $\alpha = (\alpha_1,\alpha_2,\hdots,\alpha_k)$ be a composition of $n$; define
\[
        P_\alpha = \left(\begin{array}{c|c|c|c}
\GL_{\alpha_1}(\FF_q) & * & \hdots & * \\ \hline
0 & \GL_{\alpha_2}(\FF_q) & \hdots & * \\ \hline
\vdots & \vdots & \ddots & \vdots \\ \hline
0 & 0 & \hdots & \GL_{\alpha_k}(\FF_q) \end{array}\right)
\]
to be the \emph{parabolic subgroup} of shape $\alpha$. 
If $\One_k$ is the trivial character of $\GL_k(\FF_q)$, then we have the following result.

\begin{lemma}\label{PermutationCharacterDecomposition}
Let $\lambda=(\lambda_1,\ldots,\lambda_\ell)$ be a partition of $n$; then
\[
        \Ind_{P_{\lambda}}^{\GL_n(\FF_q)}(\One)=\One_{\lambda_1}\circ\cdots\circ\One_{\lambda_\ell} = \sum_{\mu \vdash n} K_{\mu'\lambda} \chi^{\mu^{(1)}},
\]
where the $K_{\mu'\lambda}$ are the Kostka numbers (see \cite[I.6.4]{MR1354144}). In particular, $\langle \Ind_{P_{\lambda}}^G(\One), \chi^\mu \rangle = 0$ unless $\mu' \succeq \lambda$.
\end{lemma}

\subsection{Set partition combinatorics}

In this section we introduce some background and terminology regarding set partitions. A \emph{set partition} $\eta$ of $\{1,2,\hdots , n\}$ is a subset
$$\eta\subseteq \{i \frown j\mid 1\leq i<j\leq n\}$$
such that if $i\frown k,j\frown l\in \eta$, then $i=j$ if and only if $k=l$.   Let
\begin{equation}
\cS_n=\{\text{set partitions of $\{1,2,\ldots, n\}$}\}.
\end{equation}

We can represent these set partitions by an arc diagram where we line up $n$ nodes and connect the $i$th to the $j$th if $i\frown j\in \eta$.  For example,
\[
\begin{tikzpicture}[scale=.5,baseline={([yshift=-.5ex]current bounding box.center)}]
\foreach \x in {1,...,8}
	{
	\node (\x) at (\x,0) [inner sep=-1pt] {$\bullet$};
	\node at (\x,-.4) {$\scs\x$};}
	 \draw (1) to [out=45, in=135]  (4);
	\draw (3) to [out=45, in=135]  (6);
    	\draw (6) to [out=45, in=135]  (8);
\end{tikzpicture}=\{1\frown 4,3\frown 6,6\frown 8\}\in\cS_8,
\]
 but
\[
\begin{tikzpicture}[scale=.5,baseline={([yshift=-.5ex]current bounding box.center)}]
\foreach \x in {1,...,8}
	{
	\node (\x) at (\x,0) [inner sep=-1pt] {$\bullet$};
	\node at (\x,-.4) {$\scs\x$};}
	 \draw (1) to [out=45, in=135]  (4);
	  \draw (1) to [out=45, in=135]  (3);
	\draw (3) to [out=45, in=135]  (6);
    	\draw (6) to [out=45, in=135]  (8);
\end{tikzpicture}=\{1\frown 4,1\frown 3,3\frown 6,6\frown 8\}\notin \cS_8
\]
because of the pair of arcs $1 \frown 3$ and $1 \frown 4$.

We say that a set partition is \emph{nonnesting} if it contains no pair of arcs $i \frown l, j \frown k$ with $i<j<k<l$. In other words, the relative positioning of arcs
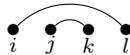
\[
\begin{tikzpicture}[scale=.5,baseline={([yshift=-.5ex]current bounding box.center)}]
\foreach \x in {1,...,4}
	\node (\x) at (\x,0) [inner sep=-1pt] {$\bullet$};
\foreach \x/\y in {1/i,2/j,3/k,4/l}
	\node at (\x,-.4) {$\scs\y$};
	 \draw (1) to [out=45, in=135]  (4);
	\draw (2) to [out=45, in=135]  (3);
\end{tikzpicture}
\]
never occurs in the arc diagram of the set partition.  Let
\begin{equation}
\cS_n^{\nn}=\{\eta\in \cS_n\mid \eta\text{ nonnesting}\}.
\end{equation}

\subsection{Unipotent subgroups of $\GL_n(\FF_q)$} \label{SectionUnipotentGroups}

A \emph{unipotent subgroup} $U\subseteq \GL_n(\FF_q)$ is a subgroup that is conjugate to a subgroup of the group of unipotent upper-triangular matrices
$$\UT_n(\FF_q)=\{g\in \GL_n(\FF_q)\mid (g-\Id)_{ij}\neq 0\text{ implies } i<j\}.$$
Since every $\mathrm{char}(\FF_q)$-group is isomorphic to a unipotent subgroup of $\GL_n(\FF_q)$ for some $n$, these subgroups can get quite messy.  We will therefore focus on the set of \emph{normal pattern subgroups}
$$\cU_n=\{U\trianglelefteq \UT_n(\FF_q)\mid  T_n\subseteq \mathbf{N}_{\GL_n(\FF_q)} (U)\},$$
 where $T_n\subseteq \GL_n(\FF_q)$ is the subgroup of diagonal matrices.   Note that $\UT_n(\FF_q)\in \cU_n$.
 
In general the elements of $\cU_n$ directly correspond to Dyck paths from $(0,0)$ to $(0,2n)$.  If $D$ is such a Dyck path then
$$U_D=\{u\in \UT_n(\FF_q)\mid (u-\Id)_{ij}\neq 0 \text{ implies $(2i-1,2j-1)$ is above $D$}\}.$$
For example,
$$U_{
\begin{tikzpicture}[scale=.2,baseline=0cm]
\foreach \x in {0,...,10} 
	{\node (\x) at (\x,0)  [inner sep=-1pt] {$\scs\bullet$};}
\foreach \x/\y/\z in {0/0/1,1/1/0,2/0/1,3/1/2,4/2/3,5/3/2,6/2/3,7/3/2,8/2/1,9/1/0}
	\draw (\x,\y)  -- (\x+1,\z);
\end{tikzpicture}} = \left[
\begin{tikzpicture}[scale=.3,baseline=1.4cm]
\foreach \x in {0,2,4,6,8,10} 
	{\node (\x) at (\x,10-\x)  [inner sep=-1pt] {$\scs\bullet$};}
	\foreach \x in {1,3,5,7,9} 
	{\node (\x) at (\x,10-\x)  [inner sep=-1pt] {$1$};}
	\foreach \x/\y in {3/9,5/9,7/9,9/9,9/7}
		\node at (\x,\y) {$\ast$};
	\foreach \x/\y in {5/7,7/7,7/5,9/5,9/3}
		\node at (\x,\y) {$0$};
\foreach \x/\y/\z/\w in {0/10/2/10,2/10/2/8,2/8/4/8,4/8/6/8,6/8/8/8,8/8/8/6,8/6/10/6,10/6/10/4,10/4/10/2,10/2/10/0}
	\draw (\x,\y)  -- (\z,\w);
\end{tikzpicture}
\right].
$$
The following special cases will be of particular importance.  
 \begin{description}
 \item[Integer compositions.]  For a composition $\alpha=(\alpha_1,\ldots, \alpha_\ell)$ of $n$, let
 $$U_\alpha =  \left[\begin{array}{c|c|c|c} \Id_{\alpha_1} & \ast & \cdots & \ast\\ \hline 0 & \Id_{\alpha_2} & \ddots & \vdots \\ \hline \vdots & \ddots & \ddots  & \ast\\ \hline 0 & \cdots & 0 & \Id_{\alpha_\ell} \end{array}\right]\subseteq \UT_n(\FF_q).$$
 Note that $U_\alpha$ is the unipotent radical of the parabolic subgroup $P_\alpha$.
 \item[Set partitions.]  For a non-nesting set partition $\eta\in \cS_n^\nn$, let
 $$U_\eta= \{u \in \UT_n(\FF_q) \mid u_{jk}=0 \text{ if  $i\leq j< k\leq l$ with $i\larc{}l\in \eta-\{j\larc{}k\}$}\}.$$
 \end{description}
We will be particularly interested in the case where $U_\eta\subseteq U_{\bl(\eta)'}$.  

\subsection{Supercharacter theories}

The notion of a supercharacter theory for a finite group $G$ was introduced by Diaconis--Isaacs in \cite{MR2373317}.  The basic idea is to treat linear combinations of irreducible characters as the ``irreducible characters" of the theory, and have a corresponding partition of $G$ whose blocks (called superclasses) are unions of conjugacy classes.  From a slightly different point of view, this gives us a Schur ring \cite{MR2989654}, and we will define a supercharacter theory from that point of view.

A \emph{supercharacter theory} $\scf(G)$ of a finite group $G$ is a subspace $\scf(G)$ of the $\CC$-space of class functions $\cf(G)$ such that $\scf(G)$ is a subalgebra of $\cf(G)$ with respect to the ring structures
\begin{enumerate}
\item[(R1)]  $(\chi\odot \psi)(g)=\chi(g)\psi(g)$, for $\chi,\psi\in \cf(G)$, $g\in G$; and
\item[(R2)]  $\dd (\chi\circ \psi)(g)=\sum_{h\in G} \chi(h)\psi(h^{-1}g)$, for $\chi,\psi\in \cf(G)$, $g\in G$.
\end{enumerate}

Each ring structure gives rise to a $\CC$-basis of orthogonal idempotents, one consisting of orthogonal characters (with respect to (R2)) and the other consists of set identifier functions that identify the superclasses (with respect to (R1)).   

\vspace{.5cm}

\noindent\textbf{Key Examples.}  
\begin{enumerate}
\item[(E1)]  The example that originally motivated the study of supercharacter theories was a supercharacter theory $\scf(\UT_n(\FF_q))$ developed by Andr\'e \cite{MR1338979}.  The superclasses and supercharacters are obtained by letting $\UT_n(\FF_q)$ act by left and right multiplication on the $\FF_q$-space $\UT_n(\FF_q)-\Id$ and its dual space.  The dimension of the resulting theory is a $(q-1)$-analogue of $|\cS_n|$.  Our original motivation for this paper was the observation that generalized Gelfand--Graev characters ``factor through" characters in $\scf(\UT_n(\FF_q))$.  However, it soon became clear that there was a better choice of supercharacter theory.
\item[(E2)]  By slightly coarsening the theory in (E1), we obtain a supercharacter theory $\scf_{nn}(\UT_n(\FF_q))$ \cite{andrews3} with dimension equal to a $(q-1)$-analogue to $|\cS_n^\nn|$.  The advantage of this theory is that the generalized Gelfand--Graev characters can be constructed by directly inducing supercharacters from this theory.
\end{enumerate}

\section{Generalized Gelfand--Graev representation construction and variations}\label{SectionGGGDefinition}

This section gives the definition and construction of the generalized Gelfand--Graev representations, and then explores some variations that give the representations by less direct means.

\subsection{A combinatorial version of Kawanaka's construction}

The generalized Gelfand--Graev representations were introduced by Kawanaka \cite{MR803335} as a source for cuspidal representations of finite groups of Lie type.  In the case of $\GL_n(\FF_q)$, the GGG characters form a basis for the space of class functions of unipotent support
$$\uscf(\GL)\cong \Lambda(X^{(1)}) \quad \text{(in the notation of Section \ref{GLnCombinatorics}).}$$
It follows that the GGG representations are indexed by integer partitions; Kawanaka constructs them from the nilpotent $\GL_n(\FF_q)$-orbits of the corresponding Lie algebra $\fkgl_n(\FF_q)$.  In this section we present a different construction that is more combinatorial in nature.

Given an integer partition $\lambda\vdash n$, we construct a unipotent subgroup $U_{\ctr(\lambda')}\subseteq \UT_n(\FF_q)$ and a linear representation $\gamma_\lambda:U_{\ctr(\lambda')}\rightarrow \GL_1(\CC)$ such that the generalized Gelfand--Graev representation $\Gamma_\lambda$ is given by
$$\Gamma_\lambda=\Ind_{U_{\ctr(\lambda')}}^{\UT_n(\FF_q)}(\gamma_\lambda).$$

\subsubsection{The unipotent subgroup $U_{\ctr(\lambda')}$}

Fix an integer partition $\lambda=(\lambda_1,\ldots, \lambda_\ell)$.   We define a permutation $\ctr\in S_{\lambda_1}$ by
$$\ctr(j)=\left\{\begin{array}{ll} \lfloor \lambda_1/2+1\rfloor+\frac{j-1}{2} &\text{if $j\notin 2\ZZ$,}\\ 
  \lceil \lambda_1/2+1\rceil -j/2 &\text{if $j\in 2\ZZ$.}\end{array}\right.$$
That is, $\ctr$ pushes all of the odd elements to the end and the even ones to the beginning; the odd elements stay in the same relative order and the even ones get placed in reverse order.   We will use $\ctr$ to permute the parts of $\lambda'$ (or equivalently the columns of the Ferrer's diagram of $\lambda$).

We are interested in three Ferrer's shapes corresponding to this partition:
\begin{enumerate}
\item[(F0)]  the usual left-justified Ferrer's shape,
\item[(F1)]  the shape obtained by centering the rows of (F0),
\item[(F2)]  the shape obtained by applying $\ctr$ to the columns of (F0) to get nearly centered rows without offsets.
\end{enumerate}
For example, if $\lambda=(4,3,2,2,1)$, then we write
$$
(\text{F0})=\begin{tikzpicture}[scale=.3,baseline=.5cm]
\foreach \x/\y in {0/0,0/1,1/1,0/2,1/2,0/3,1/3,2/3,0/4,1/4,2/4,3/4}
{
  \draw (\x,\y) +(-0.5,-0.5) rectangle ++(0.5,0.5);
}
\end{tikzpicture},\qquad
(\text{F1})=\begin{tikzpicture}[scale=.3,baseline=.5cm]
\foreach \x/\y in {1.5/0,1/1,2/1,1/2,2/2,.5/3,1.5/3,2.5/3,0/4,1/4,2/4,3/4}
{
  \draw (\x,\y) +(-0.5,-0.5) rectangle ++(0.5,0.5);
}
\end{tikzpicture},\qquad \text{and}\qquad (\text{F2})=\begin{tikzpicture}[scale=.3,baseline=.5cm]
\foreach \x/\y in {2/0,1/1,2/1,1/2,2/2,1/3,2/3,3/3,0/4,1/4,2/4,3/4}
{
  \draw (\x,\y) +(-0.5,-0.5) rectangle ++(0.5,0.5);
}
\end{tikzpicture}.
$$
We will need (F1) to define $\gamma_\lambda$; the composition $\ctr(\lambda')$ determined by the columns of (F2) gives us the subgroup $U_{\ctr(\lambda')}$ (as in Section~\ref{SectionUnipotentGroups}).

In our example,
$$U_{\ctr(\lambda')}=U_{(1,4,5,2)}=\left\{\left(\begin{array}{c|c|c|c}
\Id_1 & \ast & \ast & \ast\\ \hline
0 & \Id_4 & \ast & \ast\\ \hline
0 & 0 & \Id_5 & \ast \\ \hline
0 & 0 & 0 & \Id_2\end{array}\right)\right\}\subseteq \UT_{12}(\FF_q).$$

\begin{remark} The group $U_{\ctr(\lambda')}$ is the group $U_{1.5}$ in Kawanaka \cite{MR803335}.  There are some choices to be made in the construction of $U_{1.5}$, and our choice of column permutation to get (F2) makes these choices.
\end{remark}

\subsubsection{$\Gamma_\lambda$ from the linear representation $\gamma_\lambda$}

Consider the column reading tableau on (F1) obtained by numbering in order down consecutive half-columns.   Let $C_\lambda$ be the corresponding tableau of shape (F2) by viewing (F2) as a row shift from (F1).  In our example,
$$
\begin{tikzpicture}[scale=.4,baseline=.6cm]
\foreach \x/\y/\num in  {1.5/0/7,1/1/5,2/1/10,1/2/4,2/2/9,.5/3/2,1.5/3/6,2.5/3/11,0/4/1,1/4/3,2/4/8,3/4/12}
{
  \draw (\x,\y) +(-0.5,-0.5) rectangle ++(0.5,0.5);
  \draw (\x,\y) node {$\scscs\num$};
}
\end{tikzpicture}\qquad \text{becomes}\qquad
C_\lambda=
\begin{tikzpicture}[scale=.4,baseline=.6cm]
\foreach \x/\y/\num in  {2/0/7,1/1/5,2/1/10,1/2/4,2/2/9,1/3/2,2/3/6,3/3/11,0/4/1,1/4/3,2/4/8,3/4/12}
{
  \draw (\x,\y) +(-0.5,-0.5) rectangle ++(0.5,0.5);
  \draw (\x,\y) node {$\scscs\num$};
}
\end{tikzpicture}.
$$

Given $\lambda\vdash n$, we obtain a set-partition
$$\GGG(\lambda)=\{i\larc{}j\mid \text{$i<j$ are in the same row and consecutive columns of $C_\lambda$}\},$$
whose block sizes are the parts of $\lambda$.   In our running example,
$$\GGG(\lambda)=
\begin{tikzpicture}[scale=.5,baseline={([yshift=-.5ex]current bounding box.center)}]
\foreach \x in {1,...,12}
	{
	\node (\x) at (\x,0) [inner sep=-1pt] {$\bullet$};
	\node at (\x,-.4) {$\scs\x$};}
\foreach \x/\y in {1/3,3/8,8/12,2/6,6/11,4/9,5/10}
	 \draw (\x) to [out=45, in=135]  (\y); 
\end{tikzpicture}.$$

Fix a nontrivial homomorphism $\vartheta:\FF_q^+\rightarrow \GL_1(\CC)$.  Define
$$\begin{array}{r@{\ }c@{\ }c@{\ }c}
\gamma_\lambda: & U_{\ctr(\lambda')} & \longrightarrow & \GL_1(\CC)\\
& u & \mapsto & \dd\prod_{i\slarc{}j\in \GGG(\lambda)} \vartheta(u_{ij})
\end{array}.$$

Note that the commutator subgroup of $ U_{\ctr(\lambda')} $ is given by
$$[ U_{\ctr(\lambda')} , U_{\ctr(\lambda')} ]=\left\{u\in \UT_n(\FF_q)\;\left|\;\begin{array}{c} (u-\Id_n)_{ij}\neq 0\text{ implies $i$ is at least} \\ \text{two columns left of $j$ in $C_\lambda$}\end{array}\right.\right\}.$$
Since $[ U_{\ctr(\lambda')} , U_{\ctr(\lambda')} ]\subseteq \ker(\gamma_\lambda)$,  the function $\gamma_\lambda$ is a representation.

 Define the \emph{generalized Gelfand--Graev} representation  (or GGG representation) corresponding to the integer partition $\lambda\vdash n$ to be the induced representation
 $$\Gamma_\lambda= \Ind_{U_{\ctr(\lambda')} }^{\GL_n(\FF_q)} (\gamma_\lambda).$$

 \begin{remark}
 There are a number of choices made in this construction, but none of them matter once we induce to $\GL_n(\FF_q)$.  Section~\ref{sectioncharacterization} below explores more thoroughly how much flexibility we have in choosing a subgroup to replace $U_{\ctr(\lambda')}$ and a linear representation to replace $\gamma_\lambda$.
 \end{remark}

\subsection{A characterization of generalized Gelfand-Graev characters}\label{sectioncharacterization}

This section seeks to understand better the set of pairs
$$\{(\gamma,U)\mid U\subseteq \UT_n(\FF_q), \gamma:U\rightarrow \GL_1(\CC), \Ind_{U}^{\UT_n(\FF_q)}(\gamma)\in \QQ_{>0}\cdot \Gamma_\lambda\},$$
where $\lambda$ is a partition of $n$. While we do not completely characterize this set, we do give a strong set of sufficient conditions for inclusion that includes many of the cases we are interested in.  We begin by presenting an elementary proof for a known characterization of the generalized Gelfand--Graev characters.

As in Section \ref{SectionUnipotentGroups}, let
$$\cU_n=\{U\triangleleft \UT_n(\FF_q)\mid  T_n\subseteq \mathbf{N}_{\GL_n(\FF_q)} (U)\}.$$

For a set partition $\eta\in \cS_n$ and $U\in \cU_n$, define the function
$$\begin{array}{r@{\ }c@{\ }c@{\ }c} \gamma_\eta : & U & \longrightarrow & \GL_1(\CC)\\
& u & \mapsto & \dd\prod_{i\slarc{}j\in \eta} \vartheta(u_{ij}),
\end{array}$$
and $u_\eta\in \UT_n(\FF_q)$ by
$$(u_{\eta}-\Id_n)_{ij}=\left\{\begin{array}{@{}ll} 1 & \text{if $i\larc{}j\in \eta$,} \\ 0 & \text{otherwise.}\end{array}\right.$$

\begin{lemma} \label{GGGUnipotentCharacters}
Let $\eta\in\cS_n^{\nn}$ be a non-nesting partition. Let $U\in \cU_n$ be a subgroup containing $u_\eta$ with $[U,U] \subseteq\ker(\gamma_\eta)$ (so that $\gamma_\eta$ is a representation of $U$). If $\mu\vdash n$, then
$$\langle \Ind_{P_{\mu}}^{\GL_n(\FF_q)} (\One),  \Ind_{U}^{\GL_n(\FF_q)}(\gamma_\eta)\rangle \neq 0
\qquad\text{if and only if}\qquad \mu'\succeq \bl(\eta).$$
\end{lemma}
\begin{proof}
Note that
\begin{align*}
\langle \Ind_{P_{\mu}}^{\GL_n(\FF_q)} (\One),  \Ind_{U}^{\GL_n(\FF_q)}(\gamma_\eta)\rangle & = \dim\Big(e_{P_\mu} \CC \GL_n(\FF_q) e_{\gamma_\eta}\Big)\\
&=\#\{P_\mu g U\in P_\mu\backslash \GL_n(\FF_q)/U\mid (g^{-1}P_\mu g)\cap U \subseteq \ker(\gamma_\eta)\}.
\end{align*}
It follows from the Bruhat decomposition of $\GL_n(\FF_q)$ that any set of representatives for the cosets $S_\mu\backslash S_n$ gives a set of double coset representatives for $ P_\mu\backslash \GL_n(\FF_q)/\UT_n(\FF_q)$.    In particular,
$$\langle \Ind_{P_{\mu}}^{\GL_n(\FF_q)} (\One),  \Ind_{U}^{\GL_n(\FF_q)}(\gamma_\eta)\rangle \geq \#\{P_\mu w U\mid  S_\mu w\in S_\mu\backslash S_n,(w^{-1}P_\mu w)\cap U \subseteq \ker(\gamma_\eta)\}.$$
Let $R_\mu$ be the row-reading Young tableau of shape $\mu$.
Note that  $(w^{-1}P_\mu w)\cap U \subseteq \ker(\gamma_\eta)$ if and only if  $i<j$ in the same connected component of $\eta$ implies $w(j)$ is in a row strictly to the North of $w(i)$ in $R_\mu$.   Thus, if $\mu'\succeq \bl(\eta)$, then there exists $w\in S_n$ that satisfies $(w^{-1}P_\mu w)\cap U \subseteq \ker(\gamma_\eta)$, so
$$\langle \Ind_{P_{\mu}}^{\GL_n(\FF_q)} (\One),  \Ind_{U}^{\GL_n(\FF_q)}(\gamma_\eta)\rangle>0.$$

Conversely, suppose $\mu'\nsucceq \bl(\eta)$.  Then for any $w\in S_n$, there exists $j\larc{} k\in \eta$ such that $1+e_{w(j),w(k)}\in P_\mu$.  Thus, $1+ae_{jk}\in w^{-1}P_\mu w\cap U$ for all $a\in \FF_q^\times$, and in particular $w^{-1}P_\mu w\cap U\nsubseteq \ker(\gamma_\eta)$.  Let $u,v\in \UT_n(\FF_q)$.  Then
\begin{equation*}
(1+aue_{jk}v)= 1+\sum_{i\leq j<k\leq l} a u_{ij}v_{kl}e_{il},
\end{equation*}
where in our case we will take $v=u^{-1}$.  Since $\eta$ non-nesting,
$$\gamma_\eta(1+aue_{jk}v)=\vartheta\bigg(a\sum_{i\leq j<k\leq l\atop i\slarc{}l\in \eta} u_{ij}(u^{-1})_{kl}\bigg)=\vartheta\Big(au_{jj}(u^{-1})_{kk}\Big)=\vartheta(a)\neq 1,$$
for some $a\in \FF_q^\times$.  Thus, for $u\in \UT_n(\FF_q)$ and $w\in S_n$, we have $u^{-1}w^{-1}P_\mu wu\cap U\nsubseteq \ker(\gamma_\eta)$, so
$$\{P_\mu g U\in P_\mu\backslash \GL_n(\FF_q)/U\mid (g^{-1}P_\mu g)\cap U \subseteq \ker(\gamma_\eta)\}=\emptyset,$$
as desired.
\end{proof}

By combining Lemma~\ref{GGGUnipotentCharacters} with Lemma \ref{PermutationCharacterDecomposition}, we obtain the following corollary.

\begin{corollary}\label{corGGG}
Let $\eta\in\cS_n^{\nn}$ be a non-nesting partition, and let $U\in \cU_n$ with $u_\eta\in U$ and $[U,U] \subseteq \ker(\gamma_\eta)$. If $\mu\vdash n$, then
\[
		\langle \Ind_{U}^{\GL_n(\FF_q)}(\gamma_\eta), \chi^{\mu^{(1)}} \rangle = 0
\]
unless $\mu \succeq \bl(\eta)$, and
\[
		\langle \Ind_{U}^{\GL_n(\FF_q)}(\gamma_\eta), \chi^{\bl(\eta)^{(1)}} \rangle \neq 0.
\]
\end{corollary}

As generalized Gelfand--Graev characters are induced from unipotent subgroups, their support is contained in the set unipotent elements of $\GL_n(\FF_q)$; in fact, they form a basis for the space of unipotently supported class functions of $\GL_n(\FF_q)$.

\begin{proposition}\label{GGGUnipotentCharactersCor} For $\lambda\vdash n$,
\[
		\langle \Gamma_\lambda, \chi^{\mu^{(1)}} \rangle = 0
\]
unless $\mu \succeq \lambda$, and
\[
		\langle \Gamma_\lambda, \chi^{\lambda^{(1)}} \rangle =1.
\]
In particular, the set $\{\Gamma_\lambda \mid \lambda \vdash n\}$ is a basis for the space of unipotently supported class functions of $\GL_n(\FF_q)$.
\end{proposition}

\begin{proof} Recall that 
\[
        \Gamma_\lambda = \Ind_{U_{\ctr(\lambda')}}^{\GL_n(\FF_q)}(\gamma_\lambda) = \Ind_{U_{\ctr(\lambda')}}^{\GL_n(\FF_q)}(\gamma_{\GGG(\lambda)}).
\]        
By Corollary~\ref{corGGG}, we have that
\[
		\langle \Gamma_\lambda, \chi^{\mu^{(1)}} \rangle = 0
\]
unless $\mu \succeq \lambda$, and
\[
		\langle \Gamma_\lambda, \chi^{\lambda^{(1)}}\rangle \neq 0.
\]
This implies that the generalized Gelfand--Graev characters are linearly independent; as the unipotent conjugacy classes are also indexed by partitions of $n$, $\{\Gamma_\lambda \mid \lambda \vdash n\}$ is in fact a basis for the space of unipotently supported class functions of $\GL_n(\FF_q)$.

To see that $\langle \Gamma_\lambda, \chi^{\lambda^{(1)}}\rangle =1$, it suffices to show that $\langle \Gamma_\lambda,\Ind_{P_{\lambda'}}^{\GL_n(\FF_q)}(\One)\rangle = 1$. This follows from an argument analogous to that in the proof of Lemma~\ref{GGGUnipotentCharacters}.
\end{proof}

The following lemma allows us to determine which unipotent conjugacy classes are contained in the support of a given GGG character.

\begin{lemma} \label{GGGUnipotentClasses}
 Let $\mu\vdash n$ be a partition, and let $\alpha\vDash n$ be a composition that is a reordering of $\mu'$. If $u \in U_{\alpha}$ has Jordan type $\nu$, then $\nu \preceq \mu$.
\end{lemma}

\begin{proof} First consider the case where $\alpha = \mu'$. Since $u$ is unipotent,
\[
		\text{dim}(\text{ker}(u-1)^i) = \nu_1'+\nu_2'+...+\nu_i'.
\]
Simultaneously, $u \in U_{\alpha'}$ implies
\[
		\text{dim}(\text{ker}(u-1)^i) \geq \mu_1'+\mu_2'+...+\mu_i'
\]
for all $i$. This means that $\nu' \succeq \mu'$, hence $\nu \preceq \mu$.

Let $\alpha$ be an arbitrary a reordering of $\lambda'$, and note that
\begin{equation*}
		\Ind_{U_{\alpha}}^{\GL_n(\FF_q)}(\One)
		 = \rho_{\GL_{\alpha_1}(\FF_q)} \circ \cdots \circ \rho_{\GL_{\alpha_k}(\FF_q)}=\rho_{\GL_{\mu_1'}(\FF_q)} \circ \cdots\circ \rho_{\GL_{\mu_k'}(\FF_q)}=\Ind_{U_{\mu'}}^{\GL_n(\FF_q)}(\One),
\end{equation*}
where $\rho_G$ denotes the regular character of $G$ and $\circ$ is the commutative product on characters coming from parabolic induction (\ref{ParabolicInduction}).
In particular,
\begin{equation*}
\frac{\#\{u \in U_\alpha\mid u\text{ has Jordan type $\nu$}\}}{|U_\alpha|}=\Ind_{U_{\alpha}}^{\GL_n(\FF_q)}(\One)(u)
=\Ind_{U_{\mu'}}^{\GL_n(\FF_q)}(\One)(u)\\
=0,
\end{equation*}
unless $\nu\preceq \mu$.
\end{proof}

We can now describe the unipotent conjugacy classes on which a GGG character is nonzero.

\begin{proposition}\label{GGGUnipotentClassesCor} Let $\lambda,\mu\in\cP$ with $|\lambda|=|\mu|$; then
\begin{enumerate}
\item[(a)]  If $u \in U$ has Jordan type $\mu$, then
\[
		\Gamma_\lambda(u) = 0
\]
unless $\mu \preceq \lambda$.
\item[(b)] $\Gamma_\lambda(u_\lambda)\neq 0$.
\end{enumerate}
\end{proposition}

\begin{proof} (a) Recall that $\Gamma_\lambda$ is induced from $U_\alpha$ for a composition $\alpha$ that is a permutation of $\lambda'$. By Lemma~\ref{GGGUnipotentClasses}, we have
\[
		\Gamma_\lambda(u) = 0
\]
unless $u$ is in the unipotent conjugacy class indexed by $\mu$ for a partition $\mu \preceq \lambda$.

(b) This is a consequence of (a), along with the fact that $\{\Gamma_\lambda\mid \lambda \vdash n\}$ is a basis for the space of class functions of $G$ with unipotent support.
\end{proof}

We obtain a characterization of the generalized Gelfand--Graev characters.

\begin{theorem}[Geck--H\'ezard \cite{MR2468596}] \label{GGGCharacterization} Suppose that $f\in \uscf(\GL_n)$ and $\lambda\vdash n$.  Then $f$ satisfies
\begin{enumerate}
    \item[(G1)] $\langle f , \chi^{\mu^{(1)}} \rangle = 0$ unless $\mu \succeq \lambda$, and
    \item[(G2)] $f(u_\mu) = 0$ unless $\mu \preceq \lambda$,
\end{enumerate}
if and only if $f = c\Gamma_\lambda$ for some constant $c\in\CC$.
\end{theorem}

\begin{proof} By Propositions \ref{GGGUnipotentClassesCor} and \ref{GGGUnipotentCharactersCor}, any multiple $c\Gamma_\lambda$ for $c\in \CC$ satisfies  (G1) and (G2). Since the Gelfand--Graev characters form a basis for the class functions with unipotent support, any $f\in \uscf(\GL_n)$ satisfying (G1)  and (G2) is simultaneously of the form
\[
         \sum_{\mu \succeq \lambda} a_\mu \Gamma_\mu   \overset{(G1)}{=}f\overset{(G2)}{=} \sum_{\mu \preceq \lambda} b_\mu \Gamma_\mu,
\]
as desired.
\end{proof}

Using Lemmas  \ref{GGGUnipotentClasses} and \ref{GGGUnipotentCharacters} in addition to Theorem \ref{GGGCharacterization}, we obtain the following sufficient condition for an induced character to be a generalized Gelfand--Graev character.

\begin{corollary}\label{GGGSufficiency}
Let $U\in \cU_n$ and $\eta\in\cS_n^\nn$ be a non-nesting set partition such that 
\begin{enumerate}
\item[(a)] $u_\eta\in U$,
\item[(b)] $[U,U] \subseteq \ker(\gamma_\eta)$,
\item[(c)] $U\subseteq U_{\alpha}$ for some composition $\alpha\vDash n$ that is a re-arrangement of $\bl(\eta)'$.
\end{enumerate}
Then
$$\Ind_U^{\GL_n(\FF_q)}(\gamma_\eta) =  \frac{|U_{\ctr(\bl(\eta)')}|}{|U|}\Gamma_{\bl(\eta)}.$$
\end{corollary}

\begin{proof}
By Lemmas  \ref{GGGUnipotentClasses} and \ref{GGGUnipotentCharacters} and Theorem \ref{GGGCharacterization},
$$
\Ind_U^{\GL_n(\FF_q)}(\gamma_\eta)=c  \Gamma_{\bl(\eta)}=c\Ind_{U_{\ctr(\bl(\eta)')}}^{\GL_n(\FF_q)}(\gamma_{\bl(\eta)})$$
for some $c\in \QQ_{>0}$.  Evaluating at 1,
$$c=\frac{\Ind_U^{\GL_n(\FF_q)}(\gamma_\eta)(1)}{\Gamma_{\bl(\eta)}(1)}=\frac{|\GL_n(\FF_q)|/|U|}{|\GL_n(\FF_q)|/|U_{\ctr(\bl(\eta)')}|}=\frac{|U_{\ctr(\bl(\eta)')}|}{|U|}.$$
\end{proof}

\begin{remark}
The conditions in Corollary \ref{GGGSufficiency} are too strong to be necessary.  In particular, one can show that certain supercharacters of the usual Diaconis--Isaacs \cite{MR2373317} algebra group theory induce to generalized Gelfand--Graev characters even though they do not fit in the appropriate unipotent radical.
\end{remark}

\section{GGG characters from non-nesting supercharacters of $\UT_n(\FF_q)$}\label{SectionSupercharacters}

Since $\UT_n(\FF_q)\subseteq \GL_n(\FF_q)$ contains all of the subgroups that we are inducing from, it is natural to try to classify the representations of $\UT_n(\FF_q)$ that induce to GGG characters.   This sections shows that even though the representation theory of $\UT_n(\FF_q)$ is wild, we already know the representations that induce to GGG characters from the study of supercharacters.

We fix a nontrivial homomorphism
$$\vartheta:\mathbb{F}_q^+ \rightarrow \mathbb{C}^\times.$$

\subsection{A supercharacter theory from non-nesting set partitions}

Retaining the notation from Section \ref{SectionUnipotentGroups}, for a non-nesting set partition $\eta\in \cS_n^\nn$, the group
\[
		U_\eta = \{u \in \UT_n(\FF_q) \mid u_{jk}=0 \text{ if there exists $i\larc{}l\in \eta-\{j\larc{}k\}$ with $i\leq j< k\leq l$}\},
\]
has an associated nilpotent $\FF_q$-algebra $\mathfrak{u}_\eta=U_\eta-\Id_n$.  For any function 
$$\begin{array}{r@{\ }c@{\ }c@{\ }c}\eta^\times: &\eta &\longrightarrow &\FF_q^\times\\
& i\frown j & \mapsto & \eta_{ij},\end{array}$$
 we define a linear character $\gamma_{\eta^\times}$ of $U_\eta$ by
\[
		\gamma_{\eta^\times}(u) =  \prod_{i \frown j \in \eta} \vartheta(\eta_{ij}u_{ij}).
\]

\begin{example} For $n=8$,
\[
		U_{\begin{tikzpicture}[scale=.3]
	\foreach \x in {1,...,8}
	{	\node (\x) at (\x,0) [inner sep = -1pt] {$\bullet$};
		\node at (\x,-.7) {$\scs \x$};}
	\draw (1) to [out=60,in=120]  (4);
	\draw (3) to [out=60,in=120] (6);
	\draw (6) to [out=60,in=120]  (8);
\end{tikzpicture}} =
	\left\{\left(\begin{array}{cccccccc}
	1 & 0 & 0 & \ast & * & * & * & * \\
	0 & 1 & 0 & 0 & * & * & * & *  \\
	0 & 0 & 1 & 0 & 0 & \ast & * & *\\
	0 & 0 & 0 & 1 & 0 & 0 & * & * \\
	0 & 0 & 0 & 0 & 1 & 0 & * & * \\
	0 & 0 & 0 & 0 & 0 & 1 & 0 & \ast \\
	0 & 0 & 0 & 0 & 0 & 0 & 1 & 0 \\
	0 & 0 & 0 & 0 & 0 & 0 & 0 & 1
	\end{array}\right)\right\},
\]
and if $\eta^\times$ is given by $\eta_{14}=a, \eta_{36}=b,\eta_{68}=c$, then
\[
\gamma_{\eta^\times}(u) = \vartheta(au_{14}+bu_{36}+cu_{68}).
\]
\end{example}

With the above notation, let
$$\chi^{\eta^\times}_\nn=\Ind_{U_\eta}^{\UT_n(\FF_q)}(\gamma_{\eta^\times}).$$
These characters are in fact the supercharacters of a supercharacter theory of $\UT_n(\FF_q)$ \cite{andrews3}, which we will refer to as the \emph{non-nesting supercharacter theory} of $\UT_n(\FF_q)$.

\begin{theorem}[{\cite[Theorem 5.4]{andrews3}}] The subspace
$$\scf_\nn(\UT_n(\FF_q))=\CC\spanning\{\chi^{\eta^\times}_\nn\mid \eta\in \cS_n^\nn, \eta^\times:\eta\rightarrow \FF_q^\times\}\subseteq \cf(\UT_n(\FF_q))$$
is a supercharacter theory of $\UT_n(\FF_q)$.
\end{theorem}

\begin{remark}
In relating these supercharacters to GGG characters, our choice of $\eta^\times$ becomes immaterial; that is, we might as well send all of the arcs to $1\in \FF_q^\times$.  Therefore, we will use the convention that for $\eta\in \cS_n^\nn$,
$$\gamma_\eta(u)=\prod_{i\frown j\in \eta} \vartheta(u_{ij}).$$
\end{remark}

\subsection{From non-nesting supercharacters to GGG characters}

Let  $\alpha\vDash n$ be a composition.  An $\alpha$-column tableau $T$ is a filling of the Ferrer's shape with ordered column lengths $\alpha$, such that
\begin{enumerate}
\item[(T1)] Each number $\{1,\ldots, n\}$ appears exactly once,
\item[(T2)] For $1\leq i<j\leq \ell(\alpha)$, the entries of column $i$ are strictly less than the entries of column $j$,
\item[(T3)] If row $i$ and row $j$ in $T$ have the same length, then $i<j$ implies the first entry of row $i$ is less than the first entry of row $j$.
\end{enumerate}
Let 
$$\cT_\alpha=\{\text{$\alpha$-column tableaux}\}.$$
For example, if $\alpha=(1,5,2,4)$, then 
$$\begin{tikzpicture}[scale=.5,baseline=1cm]
\foreach \x/\y/\z in {1/0/{12},1/1/6,3/1/9,1/2/3,3/2/10,1/3/1,2/3/4,3/3/{7},0/4/2,1/4/5,2/4/{8},3/4/{11}}
{
  \draw (\x,\y) +(-0.5,-0.5) rectangle ++(0.5,0.5);
  \node at (\x,\y) {$\scscs \z$};
}
\end{tikzpicture}\in \cT_\alpha
\qquad\text{and}\qquad
\begin{tikzpicture}[scale=.5,baseline=1cm]
\foreach \x/\y/\z in {1/0/{12},1/1/3,3/1/10,1/2/6,3/2/9,1/3/1,2/3/4,3/3/{7},0/4/2,1/4/5,2/4/{8},3/4/{11}}
{
  \draw (\x,\y) +(-0.5,-0.5) rectangle ++(0.5,0.5);
  \node at (\x,\y) {$\scscs \z$};
}
\end{tikzpicture}\notin\cT_\alpha.$$
We can obtain set partitions from such tableaux by letting the rows give the connected components; more formally, we have a map
$$\begin{array}{r@{\ }c@{\ }c@{\ }c}\mathrm{sp}:  & \cT_\alpha & \longrightarrow & \cS_{|\alpha|}\\
& T & \mapsto & \dd \bigcup_{\text{row $(i_1,\ldots, i_r)$}\atop \text{of $T$}} \{i_1\frown i_2,i_2\frown i_3,\ldots, i_{r-1}\frown i_r\}. \end{array}$$
 In the above example,
$$\mathrm{sp}\left(\begin{tikzpicture}[scale=.5,baseline=1cm]
\foreach \x/\y/\z in {1/0/{12},1/1/6,3/1/9,1/2/3,3/2/10,1/3/1,2/3/4,3/3/{7},0/4/2,1/4/5,2/4/{8},3/4/{11}}
{
  \draw (\x,\y) +(-0.5,-0.5) rectangle ++(0.5,0.5);
  \node at (\x,\y) {$\scscs \z$};
}
\end{tikzpicture}
\right)
=
 \begin{tikzpicture}[scale=.5,baseline=0cm]
	\foreach \x in {1,...,12}
	{	\node (\x) at (\x,0) [inner sep = -1pt] {$\bullet$};
		\node at (\x,-.45) {$\scs \x$};}
	\foreach \x/\y in {2/5,5/8,8/11,1/4,4/7,6/9,3/10}
		\draw (\x) to [out=60,in=120] (\y);
\end{tikzpicture}.
$$
  Let
$$\cT_{\alpha}^\nn =\{T\in \cT_\alpha\mid \mathrm{sp}(T)\in \cS_{|\alpha|}^\nn\}.$$
Note that for the above tableau $T\notin \cT_\alpha^\nn$, but
$$\mathrm{sp}\left( \begin{tikzpicture}[scale=.5,baseline=1cm]
\foreach \x/\y/\z in {1/0/6,1/1/5,3/1/{10},1/2/4,3/2/{9},1/3/3,2/3/8,3/3/{12},0/4/1,1/4/2,2/4/{7},3/4/{11}}
{
  \draw (\x,\y) +(-0.5,-0.5) rectangle ++(0.5,0.5);
  \node at (\x,\y) {$\scscs \z$};
}
  \end{tikzpicture}
\right)=
 \begin{tikzpicture}[scale=.5,baseline=0cm]
	\foreach \x in {1,...,12}
	{	\node (\x) at (\x,0) [inner sep = -1pt] {$\bullet$};
		\node at (\x,-.45) {$\scs \x$};}
	\foreach \x/\y in {1/2,2/7,7/11,3/8,8/12,4/9,5/10}
		\draw (\x) to [out=60,in=120] (\y);
\end{tikzpicture}
$$
is non-nesting, so this tableau is in fact in $\cT_\alpha^\nn$.  

For $T\in \cT_{\alpha}^\nn$, let $U_T=V_T\rtimes U_{\mathrm{sp}(T)}$, where
$$ V_T=\left\{u\in \UT_n(\FF_q)\ \bigg|\  \begin{array}{@{}l@{}}(u-\Id_n)_{ij}\neq 0 \text{ implies $i\larc{}k\in \mathrm{sp}(T)$ with}\\  \text{$i<j<k$, $i$ strictly North of $j$ in $T$}\end{array}\right\}.$$
Note that while $U_T\ncong U_\alpha$, it will follow from the proof of Corollary \ref{GGGConstruction} that $|U_T|=|U_\alpha|$.

For example,
$$U_{ \begin{tikzpicture}[scale=.3,baseline=.5cm]
\foreach \x/\y/\z in {1/0/6,1/1/5,3/1/{10},1/2/4,3/2/{9},1/3/3,2/3/8,3/3/{12},0/4/1,1/4/2,2/4/{7},3/4/{11}}
{
  \draw (\x,\y) +(-0.5,-0.5) rectangle ++(0.5,0.5);
  \node at (\x,\y) {$\scscs \z$};
}
  \end{tikzpicture}}=\left\{\left(\begin{array}{cccccccccccc}
  1 & \circast &   \circast & \circast & \circast & \circast & \circast & \circast & \circast & \circast & \circast & \circast  \\
    & 1 & \ast & \ast & \ast & \ast & \circast & \circast & \circast & \circast & \circast & \circast \\
    & & 1 & \ast & \ast & \ast & 0  & \circast & \circast & \circast & \circast & \circast  \\
    & & & 1 & \ast & \ast & 0 & 0 & \circast & \circast & \circast & \circast \\
    &     & & & 1 & \ast & 0 & 0 & 0 &\circast & \circast & \circast  \\
    & & & &  & 1 & 0 & 0 & 0 & 0 & \circast & \circast \\
   &    & & &  & & 1 & \ast & \ast & \ast & \circast & \circast \\
   &  &    & & & & & 1 &  \ast & \ast & 0 & \circast   \\
   &  &  &    & & & & & 1 & 0 & 0 & 0 \\
   &  &  &    & & & & & & 1 & 0 & 0 \\
   & & &  &  &    & & & & & 1 & 0 \\
   & & & &  &  &    & & & & & 1\\
  \end{array}\right)\right\},$$
where $\ast$-entries come from $V_T$ and $\circast\ $-entries come from $U_{\mathrm{sp}(T)}$.

We get the following corollary to Corollary \ref{GGGSufficiency}.

\begin{corollary} \label{GGGConstruction}  Let $\mu\vdash n$, $\alpha\vDash n$ be a rearrangement of $\mu'$ and $T\in \cT_{\alpha}^\nn$.  Then
$$\Ind_{\UT_n(\FF_q)}^{\GL_n(\FF_q)}(\chi_\nn^{\mathrm{sp}(T)})=|V_T| \Gamma_\mu.$$
\end{corollary}
\begin{proof}
The assumption on the column values of $T$ imply that $U_{\mathrm{sp}(T)}\subseteq U_\alpha$.  Thus, by Corollary \ref{GGGSufficiency},
$$\Ind_{\UT_n(\FF_q)}^{\GL_n(\FF_q)}(\chi_\nn^{\mathrm{sp}(T)})=\frac{|U_{\ctr(\mu')}|}{|U_{\mathrm{sp}(T)}|}\Gamma_\mu.$$
We have that
$$|U_{\ctr(\mu')}|=|U_\alpha|=q^{\#\{i<j\mid \text{ column of $i$ is strictly West of column with $j$ in $T$}\}}.$$
In comparison,
$$|U_{\mathrm{sp}(T)}|=\frac{|U_\alpha|}{q^{\#\{j<k\mid i\leq j<k\leq l, i\slarc{}l\in \mathrm{sp}(T)-\{j\slarc{}k\}, \text{ column of $j$ is strictly West of column with $k$ in $T$}\}}}.$$
Consider a pair $(j,k)$ with $j$ strictly West of $k$ in $T$.  Then either the column containing $j$ or the column containing $k$ is longer.  By construction and because $\mathrm{sp}(T)$ is non-nesting, there either exists $l'>l$ such that $j\larc{}l'\in \mathrm{sp}(T)$ or $i'<i$ with $i'\larc{}k\in \mathrm{sp}(T)$.
With this observation in mind, it suffices to find a bijection between the sets
\begin{equation}\label{ABSets}
\begin{split}
\cA & =\left\{(j,k)\ \bigg|\ \begin{array}{@{}c@{}}\text{$i\leq j<k\leq l$ with $i\larc{}l\in \mathrm{sp}(T)$, $|\{i,l\}\cap \{j,k\}|=1$,}\\ \text{and $j$ is strictly West of  $k$ in $T$}\end{array}\right\}\quad\text{and}\\
\cB & = \{(i,j)\mid  \text{$i\larc{}k\in \mathrm{sp}(T)$ with $i<j<k$, $i$ strictly North of $j$ in $T$}\}.
\end{split}
\end{equation}
The obvious map $\tau:\cB\rightarrow \cA$ is given by
$$\tau(i,j)=\left\{\begin{array}{ll}
(i,j) & \text{if $i$ is strictly West of $j$ in $T$,}\\
(j,k) & \text{if $i\larc{}k\in \mathrm{sp}(T)$ and $j$ is strictly West of $k$ in $T$}.
\end{array}\right.$$
Note that this is well-defined since $(i,j)\in \cB$ implies $i\larc{}k\in \mathrm{sp}(T)$ for some $k$, so since $j$ is strictly South of $i$ it either must be either weakly West of $i$ or weakly East of $k$.

To see that $\tau$ is invertible note that any fixed $(j,k)\in \cA$ falls into exactly one of three cases:
\begin{description}
\item[Case 1.] If $i\larc{}k\in  \mathrm{sp}(T)$ with $i<j$ and $j\larc{}l\notin \mathrm{sp}(T)$ for any $l>k$, then $j$ must be South of $i$ (else, $l$ would exist since $k$ is East of $j$).  Thus, $(i,j)\in \cB$ and $\tau(i,j)=(j,k)$.
\item[Case 2.] If $i\larc{}k\notin  \mathrm{sp}(T)$ for any $i<j$ and $j\larc{}l\notin \mathrm{sp}(T)$ for some $l>k$, then $k$ must be South of $j$ (else, $i$ would exist since $j$ is West of $k$).  Thus, $(j,k)\in\cB$ and $\tau(j,k)=(j,k)$.
\item[Case 3.] If $i\larc{}k\notin  \mathrm{sp}(T)$ for some $i<j$ and $j\larc{}l\notin \mathrm{sp}(T)$ for some $l>k$, then either $i$ is North of $j$ or $i$ is South of $j$.  If $i$ is North of $j$ then $(i,j)\cB$, and $\tau(i,j)=(j,k)$.  Else, $j$ is North of $k$, so $(j,k)\cB$ and $\tau(j,k)=(j,k)$.
\end{description}
Thus, we obtain a well-defined $\tau^{-1}$ and $\tau$ is a bijection.
\end{proof}

 \begin{remark}
 Note that if $\alpha=\mu'$, then we may set $T$ equal to the column reading tableau to satisfy the hypotheses of the corollary.  From this we can conclude that we can get all the generalized Gelfand--Graev characters by inducing from non-nesting supercharacters.
 \end{remark}

\begin{lemma}
Let $\mu\vdash n$, $\alpha\vDash n$ be a rearrangement of $\mu'$, and $T\in \cT^\nn_{\alpha}$.  Then\begin{enumerate}
\item[(a)] $\gamma_{\sh(T)}$ extends to a linear character of $U_T$,
\item[(b)]
$\dd\chi^{\sh(T)}_{\nn}=|V_T|\Ind_{U_T}^{\UT_n(\FF_q)}(\gamma_{\sh(T)}).$
\end{enumerate}
\end{lemma}

\begin{proof}
(a)  To show that the support of $\gamma_{\sh(T)}$ is trivial on the commutator subgroup of $U_T$, it suffices to show that if $i<j<k$ and $i\larc{}k\in \mathrm{sp}(T)$, then either $u_{ij}=0$ or $u_{jk}=0$ for all $u\in U_T$.  In fact, by the definition of $U_{\mathrm{sp}(T)}$ it suffices to show $u_{ij}=0$ or $u_{jk}=0$ for all $u\in V_T$.

Suppose $u_{ij}\neq 0$ for some $u\in V_T$.  Then $i$ is North of $j$ in $T$.  Thus, $j$ is South of $k$ and $u_{jk}=0$ for all $u\in V_T$.  Suppose $u_{jk}\neq 0$ for some $u\in V_T$.  Then there exists $l>k$ such that $j\larc{}l\in \mathrm{sp}(T)$ and $j$ is North of $k$.  In particular, $j$ is North of $i$, so $u_{ij}=0$ for all $u\in V_T$.

(b)  Note that since $U_{\mathrm{sp}(T)}\triangleleft U_T$, we have
$$\chi^{\sh(T)}_{\nn}(u)=|V_T|\Ind_{U_T}^{\UT_n(\FF_q)}(\gamma_{\sh(T)})(u)$$
for all $u\in U_{\mathrm{sp}(T)}$.  Thus, it suffices to show that $u\in U_T-U_{\mathrm{sp}(T)}$ implies
$$\Ind_{U_T}^{\UT_n(\FF_q)}(\gamma_{\sh(T)})(u)=0.$$
Fix $u\in U_T-U_{\mathrm{sp}(T)}$.  Let $\cB$ be as in (\ref{ABSets}).  Then there exists a unique $(j,k)\in \cB$ such that
\begin{itemize}
\item $u_{jk}\neq 0$;
\item $u_{ik}\neq 0$ and $(i,k)\in \cB$ implies $i<j$; and
\item $u_{jl}\neq 0$ and $(j,l)\in \cB$ implies $k<l$.
\end{itemize}
Since $(j,k)\in \cB$, there exists $l>k$ such that $j\larc{}l\in \sh(T)$.  Then
\begin{align*}
\Ind_{U_T}^{\UT_n(\FF_q)}(\gamma_{\sh(T)})(u)&=\frac{1}{|V_T|} \sum_{g\in \UT_n(\FF_q)\atop gug^{-1}\in U_T} \gamma_{\sh(T)}(gug^{-1})\\
&=\frac{1}{|V_T|q} \sum_{t\in \FF_q}\sum_{g\in \UT_n(\FF_q)\atop gug^{-1}\in U_T} \gamma_{\sh(T)}\Big((1+te_{kl})gug^{-1}(1-te_{kl})\Big),
\end{align*}
since if $u_{ik}\neq 0$, then either $i>k$ or $i<j$; either way we have that $(1+te_{kl})gug^{-1}(1-te_{kl})\in U_T$.  The minimality of $k$ and maximality of $j$ imply that if $gug^{-1}\in U_T$, then $(gug^{-1})_{jk}=u_{jk}$.  Thus,
$$\Ind_{U_T}^{\UT_n(\FF_q)}(\gamma_{\sh(T)})(u)=\frac{1}{|V_T|q} \sum_{t\in \FF_q}\sum_{g\in \UT_n(\FF_q)\atop gug^{-1}\in U_T} \vartheta(tu_{jk})\gamma_{\sh(T)}\Big(gug^{-1}\Big)=0,$$
since $u_{jk}\neq 0$.
\end{proof}

\begin{corollary}\label{GGGMult1Construction}
Let $\mu\vdash n$ and $T$ be the column reading tableau of the standard Ferrer diagram.  Then
$$\frac{1}{|V_T|} \chi_{\nn}^{\sh(T)}(\gamma_{\sh(T)})$$
is a character of $\UT_n(\FF_q)$ that induces to $\Gamma_\mu$ (without scaling).
\end{corollary}

\begin{remark}
In defining a supercharacter theory, the supercharacters are typically defined only up to scalar multiples.  Corollary \ref{GGGMult1Construction} suggests that a good scaling for the supercharacter theory described in \cite{andrews3} is obtained by dividing by $|V_T|$.
\end{remark}

\section{Symmetric functions}\label{SectionSymmetricFunctions}

In this section we study the images of the generalized Gelfand--Graev characters under the characteristic map. In particular, we calculate the multiplicities of the irreducible characters of $\GL_n(\FF_q)$ in the generalized Gelfand--Graev characters.

\subsection{Translating between $\uscf$ and $\ucf$}

Let
$$\uscf(\GL)=\{\chi\in \cf(\GL)\mid \chi(g)\neq 0\text{ implies $g$ unipotent}\}\overset{\ch}{\longleftrightarrow} \Sym(X^{(1)})$$
be the space of unipotently supported class functions, and let
$$\ucf(\GL)= \CC\spanning\{\chi\in \Irr(\GL)\mid \langle \chi, \Ind_B^{\GL}(\One)\rangle\neq 0\} \overset{\ch}{\longleftrightarrow}\Sym(Y^{(1)})$$
be the space spanned by the unipotent characters. From the point of view of symmetric functions, there are natural projective algebra homomorphisms
\begin{equation} \label{PiXDefinition}
\begin{array}{rccc}
\pi_X:& \Sym(\GL)  & \longrightarrow & \Sym(X^{(1)})\\
& \dd\sum_{\mu\in \cP^\Phi} c_\mu(q) p_\mu & \mapsto & \dd\sum_{\lambda \in \cP} c_{\lambda^{(1)}}(q) p_{\lambda^{(1)}}
\end{array}
\end{equation}
and
\begin{equation} \label{PiYDefinition}
\begin{array}{rccc}
\pi_Y:& \Sym(\GL) & \longrightarrow & \Sym(Y^{(1)})\\
& \dd\sum_{\nu\in \cP^\Theta} c_\nu(q) p_\nu & \mapsto & \dd\sum_{\rho \in \cP} c_{\rho^{(1)}}(q) p_{\rho^{(1)}}.
\end{array}
\end{equation}
\begin{remark}
Since $\Sym(X^{(f)})$ and $\Sym(X^{(f')})$ are orthogonal for $f\neq f'$ and $\Sym(Y^{(\vphi)})$ and $\Sym(Y^{(\vphi')})$ are orthogonal for $\varphi\neq \vphi'$,
\begin{align*}
\langle \psi, g(X^{(1)})\rangle =\langle \pi_X(\psi),g(X^{(1)})\rangle \qquad &\text{for $\psi\in \Sym(\GL_n)$, $g(X^{(1)})\in \Sym_n(X^{(1)})$},\\
\langle \psi, g(Y^{(1)})\rangle =\langle \pi_Y(\psi),g(Y^{(1)})\rangle \qquad &\text{for $\psi\in \Sym(\GL_n)$, $g(Y^{(1)})\in \Sym_n(Y^{(1)})$}.
\end{align*}
\end{remark}

By composing these projections with the characteristic map, we obtain projections
\begin{align*}
        \ch^{-1}\circ\pi_X\circ\ch &: \cf(\GL) \to \uscf(\GL) \quad \text{and} \\
        \ch^{-1}\circ\pi_Y\circ\ch &: \cf(\GL) \to \ucf(\GL).
\end{align*}

The images of the power sum symmetric functions under these projections are given by the following lemma, which follows directly from \cite[IV.4.5]{MR1354144}.
\begin{lemma}\label{projectionformula} For $k\in\ZZ_{\geq 0}$, $\vphi\in \Theta$, and $f\in \Phi$,
\begin{align*}
    \pi_X(p_k(Y^{(\vphi)}))&=(-1)^{|\vphi| k-1} p_{|\vphi|k}(X^{(1)}) \quad \text{and}\\
    \pi_Y(p_k(X^{(f)}))&=\frac{(-1)^{|f|k}}{1-q^{|f|k}} p_{|f|k}(Y^{(1)}).
\end{align*}
\end{lemma}

In addition to being useful in calculations, this lemma has the following corollary.

\begin{corollary}\label{ringisomorphism}
The restricted functions $\pi_X:\Sym(Y^{(1)})\rightarrow \Sym(X^{(1)})$ and $\pi_Y:\Sym(X^{(1)})\rightarrow \Sym(Y^{(1)})$, as well as $\ch^{-1}\circ\pi_X\circ\ch : \ucf(\GL) \to \uscf(\GL)$ and $\ch^{-1}\circ\pi_Y\circ\ch : \uscf(\GL) \to \ucf(\GL)$, are ring isomorphisms. \end{corollary}

\begin{remark}
Note that the restricted functions $\pi_X$ and $\pi_Y$ are not inverses of each other, since for example
$$
\pi_X\circ\pi_Y\Big(p_k(X^{(1)})\Big)=\pi_X\Big(\frac{(-1)^k}{1-q^k}p_k(Y^{(1)})\Big)=\frac{1}{q^k-1}p_k(X^{(1)}).
$$
\end{remark}

\subsection{The multiplicities of the unipotent characters}

We now calculate the multiplicities of the irreducible unipotent characters of $\GL_n(\FF_q)$ in the generalized Gelfand--Graev characters; we do this by considering two bases of $\uscf(\GL_n)$.  

First, note that since the restricted function $\pi_Y:\Sym(X^{(1)})\rightarrow \Sym(Y^{(1)})$ is an algebra isomorphism, 
$$\Big\{\ch^{-1}\circ\pi_Y^{-1}\Big(s_\lambda(Y^{(1)})\Big)\mid \lambda\vdash n\Big\}$$
is a basis for $\uscf(\GL_n)$; by the remark following (\ref{PiYDefinition}),  for $\mu\vdash n$,
$$\Big\langle \pi_Y^{-1}\Big(s_\lambda(Y^{(1)})\Big),s_\mu(Y^{(1)})\Big\rangle = \Big\langle s_\lambda(Y^{(1)}),s_\mu(Y^{(1)})\Big\rangle=\delta_{\lambda\mu}.$$

The second basis $\{\delta_{\mu} \mid \mu \vdash n\}$ for $\uscf(\GL_n)$ is given by the indicator functions 
$$
		\delta_{\mu} (u) = \left\{\begin{array}{ll}
		1 &\quad \text{ if } u \text{ has unipotent Jordan canonical form of type $\mu$,} \\
		0 & \quad \text{ otherwise,}\end{array}\right.
$$
as in Section~\ref{GLnCombinatorics}.

Using the notation from \cite[III.7]{MR1354144}, we have transition matrices $Q(q)$, $X(0)$, and $(-1)^n z(q)z^{-1}$ given by, for $\rho\vdash n$,
\begin{equation*}
\begin{split}
p_\rho(X^{(1)}) &= \sum_{\mu\vdash n} Q(q)_{\mu\rho} \tilde P_\mu(X^{(1)}),\\
			&=\sum_{\mu\vdash n} X(0)_{\mu\rho} s_\mu(X^{(1)}),\\
			&=\sum_{\mu\vdash n}  [(-1)^n z(q)z^{-1}]_{\mu\rho}\pi_Y^{-1}\Big( p_\mu(Y^{(1)})\Big),\\
\end{split}
\end{equation*}
where the first two equalities are directly in  \cite[III.7]{MR1354144}), and the last equality follows from Lemma~\ref{projectionformula}  (where we invert the restricted function $\pi_Y:\uscf(\GL_n)\rightarrow \ucf(\GL_n)$).

\begin{lemma}\label{uscftransmatrix} The transition matrix from $\{\delta_{\mu} \mid \mu \vdash n\}$ to $\{\ch^{-1}\circ\pi_Y^{-1}(s_\lambda(Y^{(1)})) \mid \lambda \vdash n\}$ is given by
\[
		(-1)^nQ(q)^{-1}z(q)X(0)^{-T}.
\]
\end{lemma}

\begin{proof} Recall $\ch (\delta_{\mu}) = \widetilde{P}_\mu(X^{(1)})$, and the transition matrix from $\{\widetilde{P}_\mu(X^{(1)}) \mid \mu \vdash n\}$ to $\{ \pi_Y^{-1}(s_\lambda(Y^{(1)})) \mid \lambda \vdash n\}$ is
\[
        Q(q)^{-1}(-1)^nz(q)z^{-1}X(0) = (-1)^nQ(q)^{-1}z(q)X(0)^{-T},
\]
using that $z^{-1}X(0) = X(0)^{-T}$ by the orthogonality relations of the character table of the symmetric group.
\end{proof}

We now calculate the multiplicities of the irreducible unipotent characters of $\GL_n(\FF_q)$ in the generalized Gelfand--Graev characters.

\begin{theorem}\label{unipotentmultiplicity} The transition matrix from $\{\ch(\Gamma_\lambda) \mid \lambda \vdash n\}$ to $\{\pi_Y^{-1}(s_\mu(Y^{(1)})) \mid \mu \vdash n\}$ is $K(q)^T$; in particular,
\[
		\langle \Gamma_\lambda, \ch^{-1}(s_\mu(Y^{(1)})) \rangle = K_{\mu\lambda}(q).
\]
\end{theorem}

\begin{proof} Let $M$ denote the transition matrix from $\{\ch(\Gamma_\lambda) \mid \lambda \vdash n\}$ to $\{\pi_Y^{-1}(s_\mu(Y^{(1)})) \mid \mu \vdash n\}$. By Proposition~\ref{GGGUnipotentCharactersCor}, Theorem~\ref{GGGCharacterization} and Lemma~\ref{uscftransmatrix}, $M$ is the unique unipotent lower-triangular matrix such that
\[
		(-1)^nQ(q)^{-1}z(q)X(0)^{-T}M^{-1}
\]
is upper-triangular. It suffices to show that
\[
		(-1)^nQ(q)^{-1}z(q)X(0)^{-T}K(q)^{-T}
\]
is upper-triangular.

By the orthogonality relations of $X(q)$ (see \cite[III.7]{MR1354144}), we have
\begin{align*}
		(-1)^nQ(q)^{-1}z(q)X(0)^{-T}K(q)^{-T} &= (-1)^nQ(q)^{-1}z(q)X(q)^{-T} \\
		& = (-1)^nQ(q)^{-1}X(q)b(q)^{-1} \\
		& = (-1)^n\diag(q^{-n(\lambda)})K(q^{-1})^{-1}K(q)b(q)^{-1},
		\end{align*}
which is an upper-triangular matrix.
\end{proof}

As a corollary, we determine the image of $\Gamma_\lambda$ under the characteristic map.

\begin{corollary} \label{GGGInXs} For $\lambda\in \cP$, we have that
$$\ch(\Gamma_\lambda)=(-1)^{|\lambda|} Q_\lambda(X^{(1)};q),$$
where the $Q_\lambda(X^{(1)};q)$ are the Hall-Littlewood symmetric functions \cite[III.2.11]{MR1354144}.
\end{corollary}

We can also describe the image of $\Gamma_\lambda$ under the composition $\pi_Y\circ \ch$. 

\begin{corollary}  \label{GGGInYs} For $\lambda\in \cP$, we have that
$$\pi_Y\circ \ch(\Gamma_\lambda)=H_\lambda(Y^{(1)};q),$$
where the $H_\lambda(Y^{(1)};q)$ are the transformed Hall-Littlewood symmetric functions \cite[3.4]{MR2051783}
\end{corollary}

\begin{remark}
Since $\Sym(Y^{(1)})\cong \Sym(Z)\cong \Sym(X^{(1)})$,  the isomorphism $\pi_Y:\Sym(X^{(1)})\rightarrow \Sym(Y^{(1)})$  induces an automorphism $\pi:\Sym(Z)\rightarrow \Sym(Z)$.  By (\ref{VariableRelation}), $\pi$ is given by the plethysm
$$\pi(g(Z))=(-1)^n g(Z/(1-q)),\qquad \text{for $g(Z)\in \Sym_n(Z)$}.$$
From this point of view, Corollaries \ref{GGGInXs} and \ref{GGGInYs} are easily equivalent.
\end{remark}

\subsection{Multiplicities of the irreducible characters of $\GL_n(\FF_q)$}

By Theorem~\ref{unipotentmultiplicity}, we have that for $\chi^\nu$ a unipotent character, 
$$\langle \Gamma_\lambda,\chi^\nu\rangle = K_{\nu\lambda}(q).$$
A class function with unipotent support is uniquely determined by the multiplicities of the unipotent characters, thus we can use the above result to determine a formula for the multiplicities of arbitrary irreducible characters of $\GL_n(\FF_q)$ in $\Gamma_\lambda$.

For $\mu \in \cP^\Theta$, we define $\semisimple(\mu),\unipotent(\mu) \in \cP^\Theta$ by
\begin{equation*}
        \semisimple(\mu)^{(\varphi)}  = \left(1^{|\mu^{(\varphi)}|}\right) \qquad \text{and} \qquad
        \unipotent(\mu)^{(\varphi)} 
         = \left\{\begin{array}{ll}
\displaystyle{\bigcup_{\varphi \in \Theta} |\varphi|\mu^{(\varphi)}} & \quad \text{if }\varphi = \{1\}, \\
\emptyset & \quad \text{otherwise.}\end{array}\right.
\end{equation*}
\begin{remark}
If $\lambda\in \cP^\Theta$ then the Jordan decomposition of the character $\chi^\lambda$ is given by $(\chi^{\semisimple(\lambda)},\chi^{\unipotent(\lambda)})$.
\end{remark}

For $\mu$ and $\lambda$ partitions of $n$, let $\psi_\mu^\lambda$ be the value of the irreducible character $\psi^\lambda$ of $S_n$ on the conjugacy class $C_\mu$, and let $z_\mu$ be the size of the centralizer of an element of $C_\mu$. Then
\[
        s_\lambda = \sum_{\mu\vdash n} \frac{\psi_\mu^\lambda}{z_\mu} p_\mu \quad \text{and} \quad p_\mu = \sum_{\lambda\vdash n} \psi_\mu^\lambda s_\lambda.
\]
For $\nu,\mu \in P^\Theta$ with $\semisimple(\nu) = \semisimple(\mu)$, define
\[
        \psi_\mu^\nu = \prod_{\varphi \in \Theta} \psi_{\mu^{(\varphi)}}^{\nu^{(\varphi)}} \quad \text{and}\quad z_\mu = \prod_{\varphi \in \Theta} z_{\mu^{(\varphi)}}.
\]
\begin{proposition}\label{GLmultiplicities}
For $\nu\in \cP^\Theta$,
$$\langle \Gamma_\lambda,\chi^\nu\rangle=\sum_{\mu\in \cP^{\Theta}\atop \semisimple(\mu)=\semisimple(\nu)} \frac{\psi^\nu_\mu}{z_\mu} X^\lambda_{\unipotent(\mu)}(q),$$
where $X^\lambda_{\unipotent(\mu)}(q)$ is as in \cite[III.7.1]{MR1354144}.
\end{proposition}
\begin{proof}
As $\Gamma_\lambda \in \uscf(\GL)$, we have that
$$\langle \Gamma_\lambda,\chi^\nu\rangle=\langle \Gamma_\lambda,\ch^{-1}\circ\pi_X\circ\ch(\chi^\nu)\rangle.$$
By considering the change of basis matrices between the Schur functions and the power sum symmetric functions, we have that
\begin{align*}
\pi_X(s_\nu) &= \prod_{\vphi\in\Theta}\sum_{\rho \vdash |\nu^{(\varphi)}|} \frac{\psi^{\nu^{(\vphi)}}_{\rho}}{z_{\rho}}\pi_X(p_\rho(Y^{(\varphi)}))\\
&= \prod_{\vphi\in\Theta}\sum_{\rho \vdash |\nu^{(\varphi)}|} \frac{\psi^{\nu^{(\vphi)}}_{\rho}}{z_{\rho}}\pi_X(p_{|\varphi|\rho}(Y^{(1)}))\\
&= \sum_{\mu \in \cP^\Theta \atop \semisimple(\mu) = \semisimple(\nu)}\prod_{\vphi\in\Theta} \frac{\psi^{\nu^{(\vphi)}}_{\mu^{(\vphi)}}}{z_{\mu^{(\vphi)}}}\pi_X( p_{|\vphi|\mu^{(\vphi)}}(Y^{(1)}))\\
&= \sum_{\mu \in \cP^\Theta \atop \semisimple(\mu) = \semisimple(\nu)} \frac{\psi^{\nu}_{\mu}}{z_{\mu}}\pi_X(p_{\unipotent(\mu)}(Y^{(1)}))\\
&=\pi_X\bigg(
\sum_{\alpha \vdash |\nu|} \sum_{\mu\in \cP^{\Theta}\atop \semisimple(\mu)=\semisimple(\nu)} \frac{\psi^\nu_\mu}{z_\mu}  \psi^\alpha_{\unipotent(\mu)} s_\alpha\bigg).
\end{align*}
It follows that
\begin{equation*}
\langle \Gamma_\lambda,\chi^\nu\rangle
 = \sum_{\alpha \vdash |\nu|}  \sum_{\mu\in \cP^{\Theta}\atop \semisimple(\mu)=\semisimple(\nu)} \frac{\psi^\nu_\mu}{z_\mu}  \psi^\alpha_{\unipotent(\mu)} K_{\alpha\lambda}(q) =\sum_{\mu\in \cP^{\Theta}\atop \semisimple(\mu)=\semisimple(\nu)} \frac{\psi^\nu_\mu}{z_\mu} X^\lambda_{\unipotent(\mu)}(q),
\end{equation*}
with the last equality by \cite[III.7.6]{MR1354144}.
\end{proof}

When the irreducible character is a product of cuspidal characters, this equation simplifies nicely.

\begin{corollary}
If $|\nu^{(\vphi)}|\leq 1$ for all $\vphi\in \Theta$, then
$$\langle \Gamma_\lambda,\chi^\nu\rangle = X_{\unipotent(\nu)}^\lambda(q).$$
\end{corollary}

In particular, if $\lambda$ is cuspidal, then \cite[III.7.E2]{MR1354144} gives the following result.

\begin{corollary}
If $\nu\in \cP^\Theta_N$ satisfies $\nu^{(\vphi)}=(1)$ for some $\vphi\in \Theta$ with $|\vphi|=N$, then
$$\langle \Gamma_\lambda,\chi^\nu\rangle = q^{n(\lambda)}\prod_{i=1}^{\ell(\lambda)-1}(1-q^{-i}).$$
\end{corollary}

We can also consider the opposite extreme, where $\nu\in\cP^{\Theta}$ satisfies $\nu^{(\vphi)}=\emptyset$ for $|\vphi|>1$. Recall that if $\mu$ and $\nu$ are partitions, the \emph{Littlewood--Richardson coefficients} $c_{\mu\nu}^\lambda$ are defined by
\[
        s_\mu s_\nu = \sum_{\lambda \vdash n} c_{\mu\nu}^\lambda s_\lambda,
\]
as in \cite[I.9.1]{MR1354144}. For $\nu \in \cP^\Theta$ with $\nu^{(\vphi)}=\emptyset$ for $|\vphi|>1$, we define $c_\nu^\mu$ by
\[
     \prod_{\varphi \in \Theta} s_{\nu^{(\varphi)}} = \sum_{\mu \vdash |\nu|} c_\nu^\mu s_\mu.
\]

\begin{proposition}\label{LRmultiplicities}
If $\nu\in\cP^{\Theta}$ satisfies $\nu^{(\vphi)}=\emptyset$ for $|\vphi|>1$, then
$$\langle \Gamma_\lambda,\chi^\nu\rangle=\sum_{\mu\vdash|\lambda|} c_{\nu}^\mu K_{\mu\lambda}(q).$$
\end{proposition}
\begin{proof}
As $\Gamma_\lambda \in \uscf(\GL)$, we have that
$$\langle \Gamma_\lambda,\chi^\nu\rangle=\langle \Gamma_\lambda,\ch^{-1}\circ\pi_X\circ\ch(\chi^\nu)\rangle.$$
By Lemma \ref{projectionformula}, if $|\varphi|=1$ then $\pi_X\left(s_{\nu^{(\varphi)}}\left(Y^{(\varphi)}\right)\right) =\pi_X\left(s_{\nu^{(\varphi)}}\left(Y^{(1)}\right)\right)$. It follows that
\begin{align*}
\pi_X(s_\nu)&=\prod_{\vphi\in \Theta} \pi_X\left(s_{\nu^{(\vphi)}}\left(Y^{(\vphi)}\right)\right)\\
&=\pi_X\left(\prod_{\vphi\in \Theta} s_{\nu^{(\vphi)}}\left(Y^{(1)}\right)\right).
\end{align*}
By definition of the Littlewood--Richardson coefficients,
$$\prod_{\vphi\in \Theta} s_{\nu^{(\vphi)}}(Y^{(1)})=\sum_{\mu\vdash |\lambda|} c_\nu^\mu s_\mu.$$
Thus we have that
$$\langle \Gamma_\lambda,\chi^\nu\rangle=\langle \Gamma_\lambda,\prod_{\vphi\in \Theta} s_{\nu^{(\vphi)}}(Y^{(1)})\rangle=\sum_{\mu\vdash|\lambda|} c_{\nu}^\mu K_{\mu\lambda}(q),$$
as desired.
\end{proof}

\begin{remark}
Proposition~\ref{GLmultiplicities} and Proposition~\ref{LRmultiplicities} give us two different formulas for the multiplicities of the irreducible characters that correspond to $\nu\in\cP^{\Theta}$ satisfying $\nu^{(\vphi)}=\emptyset$ for $|\vphi|>1$. The connection between these formulas is that the Littlewood-Richardson coefficients can be written in terms of the character table of the symmetric group by first rewriting a product of Schur functions in terms of power sum symmetric functions, then multiplying, and finally translating back to Schur functions.
\end{remark}

\def\cprime{$'$} \def\cprime{$'$}


\begin{thebibliography}{10}

\bibitem{MR2880223}
M.~Aguiar, C.~Andr{\'e}, C.~Benedetti, N.~Bergeron, Z.~Chen, P.~Diaconis,
  A.~Hendrickson, S.~Hsiao, I.~M. Isaacs, A.~Jedwab, K.~Johnson, G.~Karaali,
  A.~Lauve, T.~Le, S.~Lewis, H.~Li, K.~Magaard, E.~Marberg, J.-C. Novelli,
  A.~Pang, F.~Saliola, L.~Tevlin, J.-Y. Thibon, N.~Thiem, V.~Venkateswaran,
  C.~R. Vinroot, N.~Yan, and M.~Zabrocki.
\newblock Supercharacters, symmetric functions in noncommuting variables, and
  related {H}opf algebras.
\newblock {\em Adv. Math.}, 229(4):2310--2337, 2012.

\bibitem{MR3117506}
M.~Aguiar, N.~Bergeron, and N.~Thiem.
\newblock Hopf monoids from class functions on unitriangular matrices.
\newblock {\em Algebra Number Theory}, 7(7):1743--1779, 2013.

\bibitem{MR1338979}
C.~A.~M. Andr{\'e}.
\newblock Basic characters of the unitriangular group.
\newblock {\em J. Algebra}, 175(1):287--319, 1995.

\bibitem{MR3323980}
Carlos A.~M. Andr{\'e}, Pedro~J. Freitas, and Ana~Margarida Neto.
\newblock A supercharacter theory for involutive algebra groups.
\newblock {\em J. Algebra}, 430:159--190, 2015.

\bibitem{andrews3}
S.~Andrews.
\newblock The {H}opf monoid on nonnesting supercharacters of pattern groups.
\newblock 2014.
\newblock Preprint available at \url{http://arxiv.org/abs/1405.5480}.

\bibitem{andrews4}
S.~Andrews.
\newblock The irreducible unipotent modules of the finite general linear groups
  via tableaux.
\newblock 2015.
\newblock Preprint available at \url{http://arxiv.org/abs/1502.06542}.

\bibitem{MR3295975}
Scott Andrews.
\newblock Supercharacters of unipotent groups defined by involutions.
\newblock {\em J. Algebra}, 425:1--30, 2015.

\bibitem{MR2373317}
P.~Diaconis and I.~M. Isaacs.
\newblock Supercharacters and superclasses for algebra groups.
\newblock {\em Trans. Amer. Math. Soc.}, 360(5):2359--2392, 2008.

\bibitem{MR3313496}
Richard Dipper and Qiong Guo.
\newblock Irreducible constituents of minimal degree in supercharacters of the
  finite unitriangular groups.
\newblock {\em J. Pure Appl. Algebra}, 219(7):2559--2580, 2015.

\bibitem{MR2468596}
M.~Geck and D.~H{\'e}zard.
\newblock On the unipotent support of character sheaves.
\newblock {\em Osaka J. Math.}, 45(3):819--831, 2008.

\bibitem{MR0072878}
J.~A. Green.
\newblock The characters of the finite general linear groups.
\newblock {\em Trans. Amer. Math. Soc.}, 80:402--447, 1955.

\bibitem{MR2051783}
M.~Haiman.
\newblock Combinatorics, symmetric functions, and {H}ilbert schemes.
\newblock In {\em Current developments in mathematics, 2002}, pages 39--111.
  Int. Press, Somerville, MA, 2003.

\bibitem{MR2989654}
A.~O.~F. Hendrickson.
\newblock Supercharacter theory constructions corresponding to {S}chur ring
  products.
\newblock {\em Comm. Algebra}, 40(12):4420--4438, 2012.

\bibitem{MR803335}
N.~Kawanaka.
\newblock Generalized {G}el\cprime fand-{G}raev representations and {E}nnola
  duality.
\newblock In {\em Algebraic groups and related topics ({K}yoto/{N}agoya,
  1983)}, volume~6 of {\em Adv. Stud. Pure Math.}, pages 175--206.
  North-Holland, Amsterdam, 1985.

\bibitem{MR742472}
George Lusztig.
\newblock {\em Characters of reductive groups over a finite field}, volume 107
  of {\em Annals of Mathematics Studies}.
\newblock Princeton University Press, Princeton, NJ, 1984.

\bibitem{MR1354144}
I.~G. Macdonald.
\newblock {\em Symmetric functions and {H}all polynomials}.
\newblock Oxford Mathematical Monographs. The Clarendon Press, Oxford
  University Press, New York, second edition, 1995.
\newblock With contributions by A. Zelevinsky, Oxford Science Publications.

\bibitem{MR3081621}
Jay Taylor.
\newblock On unipotent supports of reductive groups with a disconnected centre.
\newblock {\em J. Algebra}, 391:41--61, 2013.

\bibitem{yan}
N.~Yan.
\newblock Representations of finite unipotent linear groups by the method of
  clusters.
\newblock 2010.
\newblock Preprint available at \url{http://arxiv.org/abs/1004.2674}.

\bibitem{MR643482}
A.~V. Zelevinsky.
\newblock {\em Representations of finite classical groups}, volume 869 of {\em
  Lecture Notes in Mathematics}.
\newblock Springer-Verlag, Berlin, 1981.
\newblock A Hopf algebra approach.

\end{thebibliography}
\end{document}